\documentclass[11pt,twoside]{amsart}
\pdfoutput=1 

\usepackage{latexsym}
\usepackage{amsmath}
\usepackage{amsfonts}
\usepackage{amssymb}
\usepackage[colorlinks=true,linkcolor=black,anchorcolor=black,citecolor=black,filecolor=black,menucolor=black,runcolor=black,urlcolor=black]{hyperref}

\usepackage{mathrsfs}
\usepackage{float}
\usepackage{tikz}

\tikzset{
	element/.style={circle,fill=black,scale=0.6,label distance=0.15}
}

\usetikzlibrary{calc}

\usepackage{xcolor}
\usepackage[mathscr]{euscript}

\usepackage{enumitem} 

\newtheorem{theorem}{Theorem}[section]
\newtheorem{lemma}[theorem]{Lemma}
\newtheorem{corollary}[theorem]{Corollary}
\newtheorem{prop}[theorem]{Proposition}

\numberwithin{equation}{section} 

\newcommand{\e}{\emph}
\newcommand{\cl}{{\rm cl}}

\newcommand{\delete}{\backslash}

\makeatletter
\newdimen\p@renwd \setbox0=\hbox{\phantom B} \p@renwd=\wd0
\def\bordersquare#1{\begingroup \m@th
	\setbox0=\vbox{\def\cr{\crcr\noalign{\kern2pt\global\let\cr=\endline}}
		\ialign{$##$\hfil\kern2pt\kern\p@renwd&\thinspace\hfil$##$\hfil
			&&\quad\hfil$##$\hfil\crcr
			\omit\strut\hfil\crcr\noalign{\kern-\baselineskip}
			#1\crcr\omit\strut\cr}}
	\setbox2=\vbox{\unvcopy0 \global\setbox1=\lastbox}
	\setbox2=\hbox{\unhbox1 \unskip \global\setbox1=\lastbox}
	\setbox2=\hbox{$\kern\wd1\kern-\p@renwd \left[ \kern-\wd1
		\global\setbox1=\vbox{\box1\kern2pt}
		\vcenter{\kern-\ht1 \unvbox0 \kern-\baselineskip} \,\right]$}
	\null\;\vbox{\kern\ht1\box2}\endgroup}
\makeatother

\newcommand{\cW}{\mathcal{W}}

\newcommand{\kl}{$(k,\ell)$}

\usepackage[T1]{fontenc}
\usepackage[utf8]{inputenc}

\title{A generalisation of uniform matroids}
\author[G.Drummond]{George Drummond}
\address{School of Mathematics and Statistics \\
    University of Canterbury \\
    Christchurch, New Zealand}
\email{george.drummond@pg.canterbury.ac.nz}


\begin{document}
\thispagestyle{empty}

\begin{abstract}
	A matroid is uniform if and only if it has no minor \linebreak isomorphic to $U_{1,1}\oplus U_{0,1}$ and is paving if and only if it has no \linebreak minor isomorphic to $U_{2,2}\oplus U_{0,1}$. This paper considers, more generally, when a matroid $M$ has no $U_{k,k}\oplus U_{0,\ell}$-minor for a fixed pair of positive integers $(k,\ell)$. Calling such a matroid \emph{$(k,\ell)$-uniform}, it is shown that this is equivalent to the condition that every rank-$(r(M)-k)$ flat of $M$ has nullity less than $\ell$. Generalising a result of Rajpal, we prove that for any pair $(k,\ell)$ of positive integers and prime power $q$, only finitely many simple cosimple $GF(q)$-representable matroids are \kl-uniform. Consequently, if Rota's Conjecture holds, then for every prime power $q$, there exists a pair $(k_q,\ell_q)$ of positive integers such that every excluded minor of $GF(q)$-representability is $(k_q,\ell_q)$-uniform. We also determine all binary $(2,2)$-uniform matroids and show the maximally $3$-connected members to be $Z_5\delete t, AG(4,2), AG(4,2)^*$ and a particular self-dual matroid $P_{10}$. Combined with results of Acketa and Rajpal, this completes the list of binary $(k,\ell)$-uniform matroids for which $k+\ell\leq 4$.
\end{abstract}

\maketitle

\section{Introduction}\label{sec: intro}
A matroid $M$ is \emph{paving} if every rank-$(r(M)-2)$ flat is independent, or equivalently, if $M|H$ is uniform for every hyperplane $H$ of $M$. Thus, in a natural sense, paving matroids are close to uniform. In this paper, we generalise this observation and describe  a two-parameter property of matroids that captures just how close to uniform a given matroid is.

For positive integers $k$ and $\ell$, we define a matroid to be \e{$(k,\ell)$-uniform} if it has no minor isomorphic to $U_{k,k}\oplus U_{0,\ell}$. It is easy to show that a matroid is $(1,1)$-uniform precisely if it is uniform and is $(2,1)$-uniform precisely if it is paving. It is also evident that all matroids are \kl-uniform for some $(k,\ell)$ pair and that if $M$ is \kl-uniform, then it is $(k',\ell')$-uniform for all $k'\geq k$ and $\ell'\geq \ell$. Furthermore, an easy duality argument shows that a matroid is \kl-uniform if and only if its dual is $(\ell,k)$-uniform. We will use all these facts freely. The first main result of this paper, proved in Section~\ref{sec: finiteness}, concerns representability of \kl-uniform matroids.
\begin{theorem}\label{thm: finiteness}
	For every pair $(k,\ell)$ of positive integers and every prime power $q$, only finitely many simple cosimple $GF(q)$-representable matroids are \kl-uniform.
\end{theorem}
Note that both the simple and cosimple requirements in this theorem are necessary, as the uniform matroids $U_{1,n}$ and $U_{n-1,n}$ are representable over every field for all $n\geq 1$. The following is an interesting corollary of this theorem:

\begin{corollary}
	For every prime power $q$, the set of excluded minors for $GF(q)$-representability is finite if and only if for some fixed pair $(k_q,\ell_q)$ of positive integers, every such excluded minor is $(k_q,\ell_q)$-uniform.
\end{corollary}
To illustrate, for $q\leq 4$, every excluded minor of $GF(q)$-representability is $(2,1)$-uniform, that is, paving. As Geelen, Gerards and Whittle have announced a proof of Rota's Conjecture~\cite{Geelen et al}, it would seem that such $(k_q,\ell_q)$ pairs exist for all $q$. If well behaved, these bounds may offer improved methods for explicitly determining the excluded minors of $GF(q)$-representability.

By applying duality to the lists of binary $(2,1)$-uniform and $(3,1)$-uniform matroids of Acketa~\cite{Acketa} and Rajpal~\cite{Rajpal k-paving} respectively, one may explicitly list all binary $(1,2)$-uniform and $(1,3)$-uniform matroids. These results concern binary \kl-uniform matroids such that $k+\ell\leq 4$. We complete this picture in Sections~\ref{sec: not 3conn} and~\ref{sec: 3conn} by determining the binary $(2,2)$-uniform matroids. The most difficult part of the characterisation is in establishing the following result, the proof of which appears in Section~\ref{sec: 3conn}: 
\begin{theorem}\label{thm: 3-connected binary (1,2)}
	The $3$-connected binary $(2,2)$-uniform matroids are precisely the $3$-connected minors of $Z_5\delete t, P_{10},AG(4,2)$, and $AG(4,2)^*$.
\end{theorem}
Here, $Z_5\delete t$ is the tipless binary $5$-spike, $AG(4,2)$ is the rank-$5$ affine \linebreak geometry, and $P_{10}$ is the rank-$5$ binary matroid represented by the matrix of Figure~\ref{fig: P10}. It is easily seen that $P_{10}$ is self-dual and that $P_{10}/5\delete 10\cong M(\cW_4)$. Moreover, by pivoting, one can show that $P_{10}/8\cong Z_4$. A further description of $P_{10}$ is given in Section~\ref{sec: 3conn}. The terminology throughout will follow Oxley~\cite{Oxley2011} unless otherwise specified. 

\begin{figure}[ht]
	\[
	\bordersquare{&1&2&3&4&5&6&7&8&9&10\cr
		&1&0&0&0&0&1&0&0&1&1\cr
		&0&1&0&0&0&1&1&0&0&1\cr
		&0&0&1&0&0&0&1&1&0&1\cr
		&0&0&0&1&0&0&0&1&1&0\cr
		&0&0&0&0&1&1&1&1&0&0\cr
	}
	\]
	\caption{A binary representation of $P_{10}$.}
	\label{fig: P10}
\end{figure}

The \emph{nullity} of a set $X$ in a matroid $M$ is $|X|-r_M(X)$. We conclude this section by proving a characterisation of \kl-uniform matroids in terms of nullity of flats that will be treated as an alternate definition.

\begin{prop}\label{prop: nullity def}
	A matroid $M$ is \kl-uniform if and only if every rank-$(r(M)-k)$ flat of $M$ has nullity less than $\ell$.
\end{prop}
\begin{proof}
	Suppose first that $M$ is not \kl-uniform. Then $M$ has an independent set $X$ and coindependent set $Y$ such that $M/X\delete Y\cong U_{k,k}\oplus U_{0,\ell}$ and $r_M(X)=r(M)-k$. Letting $Z$ denote the $\ell$ loops of $M/X\delete Y$, every element of $Z$ must be in the closure of $X$ in $M$. Thus, $\cl_M(X)$ is a rank-$(r(M)-k)$ flat of $M$ with nullity at least $\ell$.
	
	For the converse, suppose $M$ has a rank-$(r(M)-k)$ flat $F$ of nullity at least $\ell$. Contracting any basis for $F$ achieves a rank-$k$ matroid with at least $\ell$ loops. An appropriate restriction then yields a $U_{k,k}\oplus U_{0,\ell}$-minor.
\end{proof}

\section{Finiteness over \texorpdfstring{$GF(q)$}{GF(q)}}\label{sec: finiteness}
In this section, we prove Theorem~\ref{thm: finiteness} by showing that, for any $(k,\ell)$ pair and prime power $q$, there exists a constant $f(k,\ell,q)$ bounding the rank of any simple cosimple \kl-uniform $GF(q)$-representable matroid. We will require the following result, a rewording of \cite[Proposition 8]{Rajpal k-paving}.
\begin{lemma}\label{lem: Rajpal N(k,q)}
	For $k\geq 2$, let $M$ be a simple $(k,1)$-uniform matroid such that $M$ is $GF(q)$-representable and $r^*(M)>q$. Then there is a constant $g(k,q)$ such that $r(M)\leq g(k,q)$.
\end{lemma}
We will also make use of the following easy lemma.
\begin{lemma}\label{lem: simple}
	If $M$ is a matroid of rank $r\geq 2$, then $M$ is simple if and only if it is $(r-1,1)$-uniform.
\end{lemma}

Theorem~\ref{thm: finiteness} is then a direct consequence of the following:

\begin{prop}\label{prop: f(k,l,q)}
For every pair $(k,\ell)$ of positive integers and every prime power $q$, there is a constant $f(k,\ell,q)$ such that if $M$ is a simple cosimple $GF(q)$-representable \kl-uniform matroid, then $r(M)\leq f(k,\ell,q)$. Moreover, if $k\geq 2$, then $$f(k,\ell,q)\leq \max\{f(k-1,\ell+1,q),f(1,\ell,q)+(k-1)\}.$$
\end{prop}
\begin{proof}
Fixing $q$, we perform induction on $k$. The base case $k=1$ we split into two parts.  Firstly, if $k=\ell=1$, then $M$ is uniform and $r(M)\leq q-1$, as otherwise $M$ has a $U_{q,q+2}$ minor, a contradiction to $GF(q)$-representability. If $k=1$ and $\ell \geq2$, then by the dual of Lemma~\ref{lem: Rajpal N(k,q)}, either $r(M)\leq q$ or $r^*(M)\leq g(\ell,q)$. In the latter case, since $M$ is cosimple, $M^*$ is a restriction of $PG(g(\ell,q)-1,q)$. It follows that  $r(M)\leq (q^{g(\ell,q)}-1)/(q-1) - g(\ell,q)$. Thus, $f(1,\ell,q)$ exists for all $\ell\geq 1$.

Now suppose $k\geq 2$ and that $f(k',\ell',q)$ exists for all $k'<k$, $\ell' \geq 1$. If $M$ is also $(k-1,\ell+1)$-uniform, then $r(M)\leq f(k-1,\ell+1,q)$ by induction. Otherwise, $M$ has a rank-$(r(M)-k+1)$ flat $F$ with nullity at least $\ell+1$. As $M$ is \kl-uniform, $M|F$ is $(1,\ell)$-uniform. By duality, $(M|F)^*$ is $(\ell,1)$-uniform and, as $r^*(M|F)\geq \ell+1$, it follows Lemma~\ref{lem: simple} that $M|F$ is cosimple. Thus, $M|F$ is a simple cosimple binary $(1,\ell)$-uniform matroid. By induction, $r(M|F)\leq f(1,\ell,q)$ and hence $r(M)\leq f(1,\ell,q)+(k-1)$. We conclude that $f(k,\ell,q)$ exists and $f(k,\ell,q)\leq \max\{f(k-1,\ell+1,q),f(1,\ell,q)+(k-1)\}$.
\end{proof}
 
The next two results are extracted from \cite{Acketa} and \cite{Rajpal k-paving} respectively.
\begin{lemma}\label{lem: f112 and f022}
	$f(2,1,2)=f(1,2,2)= 4$.
\end{lemma}

\begin{lemma}\label{lem: f212 abd f032}
	$f(3,1,2)=5$ and $f(1,3,2)=11$.
\end{lemma}
It follows Proposition~\ref{prop: f(k,l,q)}, that a simple cosimple binary $(2,2)$-uniform matroid has rank at most $11$. A more helpful bound, that will be instrumental in our determination of the binary $(2,2)$-uniform matroids is the following:
\begin{lemma}\label{lem: r,r* bound}
	Let $M$ be a simple cosimple binary matroid. If $M$ is \linebreak $(2,2)$-uniform, then $\min\{r(M),r^*(M)\}\leq 5$.
\end{lemma}
\begin{proof}
	If $M$ is $(1,3)$-uniform, then $M^*$ is $(3,1)$-uniform and $r^*(M)\leq 5$ by Lemma~\ref{lem: f212 abd f032}. Otherwise, $M$ has a hyperplane $H$ of nullity at least $3$. The matroid $M|H$ is simple and $(1,2)$-uniform. Furthermore, its dual is $(2,1)$-uniform and has rank at least $3$. Thus, by Lemma~\ref{lem: simple}, $M|H$ is cosimple. It follows Lemma~\ref{lem: f112 and f022} that $r(M|H)\leq 4$ and, consequently, $r(M)\leq 5$.
\end{proof}
The remaining two sections will make frequent use of the fact that a matroid $M$ is $(2,2)$-uniform if and only if the union of any pair of circuits of $M$ has rank at least $r(M)-1$.
\section{The \texorpdfstring{$(2,2)$}{(2,2)}-uniform matroids that are not \texorpdfstring{$3$}{3}-connected}\label{sec: not 3conn}
In this section we describe all $(2,2)$-uniform matroids which are not \linebreak $3$-connected and explicitly list those that are  binary. The following results contain some redundancy but have been chosen for their clarity and to emphasise links to paving matroids. A matroid $M$ is \emph{sparse paving} if both $M$ and $M^*$ are paving, or equivalently, if $M$ is both $(2,1)$ and $(1,2)$-uniform.

\begin{prop}\label{prop: discon 22-uniform}
	Let $M$ be a disconnected matroid. Then $M$ is \linebreak $(2,2)$-uniform if and only if
	\begin{enumerate}[label=\text{\emph{(\roman*)}},leftmargin=*,align=left]
		\item $M$ or $M^*$ is paving; or
		\item $M\cong M_p\oplus U_{0,1}$ or $M\cong M_p^*\oplus U_{1,1}$, where $M_p$ is a paving matroid; or
		\item $M\cong M_p \oplus U_{1,2}$, where $M_p$ is a sparse paving matroid.
	\end{enumerate}
\end{prop}
\begin{proof}
The disconnected matroids of type (i), (ii) and (iii) are easily seen to be $(2,2)$-uniform. To see that there are no others, let $M$ be a disconnected $(2,2)$-uniform matroid. If $M$ has a loop $l$, then $M\delete l$ is certainly paving and (ii) holds. Otherwise, by duality, we may assume that $M$ has no loops or coloops. It follows that if $r(M)\leq 2$ or $r^*(M)\leq 2$, then (i) holds. Hence, we may also assume that $r(M),r^*(M)\geq 3$. Now, if every component of $M$ has rank, corank at least two, then each component contains at least two circuits and the union of any two such circuits has rank less than $r(M)-1$, a contradiction to the $(2,2)$-uniform property. Thus, up to duality, $M$ has at least one rank-$1$ component $M_1$. If $|E(M_1)|\geq 3$, then by the $(2,2)$-uniform property, $r(M)\leq 2$, a contradiction. Thus, $M_1\cong U_{1,2}$. It then follows easily from the $(2,2)$-uniform property that $M\delete E(M_1)$ is both $(2,1)$-uniform and $(1,2)$-uniform. In particular, (iii) is satisfied.
\end{proof}
We follow Oxley~\cite{Oxley2011} in using $P(M_1,M_2)$ to denote the parallel connection of matroids $M_1$ and $M_2$ across some common basepoint.
\begin{prop}\label{prop: conn 22-uniform}
	Let $M$ be a connected matroid that is not $3$-connected. Then $M$ is $(2,2)$-uniform if and only if
	
	\begin{enumerate}[label=\text{\emph{(\roman*)}},leftmargin=*,align=left]
		\item $M$ or $M^*$ is  paving; or
		\item $M$ or $M^*$ has rank $3$ and no parallel class of size more than two; or
		\item $M$ has a parallel or series pair $\{p,p'\}$ such that $M\delete p/p'$ is sparse paving; or
		\item $M= P(N,U_{2,4})\delete p$, where $N$ is a connected matroid such that $N/p$ and $N^*/p$ are paving.
		
	\end{enumerate}
\end{prop}
\begin{proof}
	It is straightforward to show that all matroids of type (i)-(iv) are $(2,2)$-uniform. To see that this list is complete, let $M$ be a connected $(2,2)$-uniform matroid that is not $3$-connected. If $M$ has rank or corank at most $3$, then it is easily seen to satisfy (i) or (ii). Thus, we may assume that $r(M),r^*(M)\geq 4$. Suppose now that, up to duality, $M$ has a parallel pair $\{p,p'\}$ and let $N=M\delete p/p'$. If there exists a circuit $C$ of $N$ of rank at most $r(N)-2$, then as $C$ or $C\cup p'$ is a circuit of $M$, it follows that $C\cup \{p,p'\}$ contains two circuits of $M$ whose union has rank at most $r(N)-1=r(M)-2$. Similarly, if there exists a pair of circuits $C_1,C_2$ of $N$ such that $r_N(C_1\cup C_2)\leq r(N)-1$, then $C_1\cup C_2\cup\{p,p'\}$ contains two circuits of $M$ whose union has rank at most $r(M)-2$. Both situations contradict the fact that $M$ is $(2,2)$-uniform. Hence, $N$ is sparse paving and (iii) holds. Otherwise, $M$ has no parallel or series pairs and we may assume that $M=P(M_1,M_2)\delete p$, for some connected matroids $M_1, M_2$ each having at least three elements and rank, corank at least two. If $r(M_1), r(M_2)\geq 3$, then by the $(2,2)$-uniform property, each of $M_1\delete p$ and $M_2\delete p$ contains at most one circuit. As each $M_i$ is connected, it follows that for $i\in \{1,2\}$, $M_i\delete p$ is a circuit and $r^*(M)\leq3$, a contradiction. Thus, without loss of generality, $r(M_1)= 2$ and $r(M)=r(M_2)+1$. If $|E(M_1)|\geq 5$, then $E(M_1)-p$ contains two triangles of $M$, and by the $(2,2)$-uniform property, $r(M)\leq 3$, a contradiction. Thus, $M_1\cong U_{2,4}$. Now let $T=E(M_1)-p$. By the $(2,2)$-uniform property, $r_M(C\cup T)\geq r(M)-1=r(M_2)$ for every circuit $C$ of $M_2$. It follows that every circuit of $M_2$ containing $p$ must have rank at least $r(M_2)-1$ and every circuit avoiding $p$ has rank at least $r(M_2)-2$. Thus, $M_2/p$ is paving. Also by the $(2,2)$-uniform property, every pair of circuits of $M_2\delete p$ must span. We conclude that $M_2/p$ and $M_2^*/p=(M_2\delete p)^*$ are paving and that (iv) holds.
\end{proof}
Restricting our attention to binary matroids, we may ignore case (iv) of Proposition~\ref{prop: conn 22-uniform} as such matroids have a $U_{2,4}$-minor. We then achieve the following list by combining Propositions~\ref{prop: discon 22-uniform} and~\ref{prop: conn 22-uniform} with Acketa's list~\cite{Acketa} of binary paving matroids. Note that, as $M(\cW_3), F_7, F_7^*$ and $AG(3,2)$ have transitive automorphism groups, any parallel connections of these matroids and $U_{2,3}$ are free of reference to a specific basepoint. The matroid $S_8$ is isomorphic to the unique non-tip deletion of the binary $4$-spike $Z_4$.
\begin{corollary}\label{cor: binary 22 list}
	The following matroids and their duals are all the binary $(2,2)$-uniform matroids that are not $3$-connected.
	\begin{enumerate}[label=\text{\emph{(\roman*)}},leftmargin=*,align=left]
		\item The matroids of rank at most $1$ other than $U_{0,1},U_{1,1},U_{1,2},U_{1,3}$;
		\item the non-simple rank-$2$ binary matroids with at most one loop;
		\item the looples, non-simple rank-$3$ binary matroids with every parallel class of size at most $2$; 
		\item $M_p\oplus U_{0,1}$ and $M_p\oplus U_{1,2}$, for $M_p$ in $\{M(\cW_3), F_7, F_7^*, AG(3,2)\}$;
		\item $P(Z_4,U_{2,3})\delete t$ and $P(S_8,U_{2,3})\delete t$, where $t$ is the tip of $Z_4$;
		\item $P(F_7,U_{2,3})\delete p$ and $P(AG(3,2),U_{2,3})\delete p$; and
		\item $P(M_p,U_{2,3})$ for $M_p$ in $\{M(\cW_3),F_7, F_7^*,AG(3,2)\}$.
	\end{enumerate}
\end{corollary}

\section{The \texorpdfstring{$3$}{3}-connected binary \texorpdfstring{$(2,2)$}{(2,2)}-uniform matroids}\label{sec: 3conn}
In this section we prove Theorem~\ref{thm: 3-connected binary (1,2)}, and in doing so, complete the determination of the binary $(2,2)$-uniform matroids. We also remark that two of the important matroids of this section, $P_9$ and $L_{10}$, arise as graft matroids. A \emph{graft} \cite{Seymour} is a pair $(G,\gamma)$ where $G$ is a graph and $\gamma$ is a subset of $V(G)$ thought of as the \emph{coloured} vertices. The associated \emph{graft matroid} is the vector matroid of the matrix obtained by adjoining the incidence vector of the set $\gamma$ to the vertex-edge incidence matrix of $G$. We follow \cite{Oxley 1987} in using $P_9$ to denote the simple binary extension of $M(\cW_4)$ represented by the matrix of Figure~\ref{fig: P9}. This is isomorphic to the graft of $\cW_4$ in which the hub vertex and three of the four rim vertices are coloured. By considering the representation of the matroid $P_{10}$ given in Figure~\ref{fig: P10}, we see that $P_{10}$ arises as a single-element coextension of $P_9$. In fact, it is routine (if tedious) to verify that $P_{10}$ is the $3$-sum of $P_9$ and $F_7$ across any of the four triangles of $P_9$ other than $\{1,4,8\}$ and $\{3,4,7\}$. Up to isomorphism, there are two other simple binary extensions of $M(\cW_4)$, namely $M(K_5\delete e)$ and $M^*(K_{3,3})$.

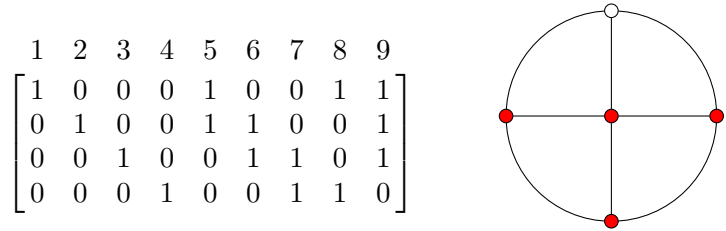
\begin{figure}[ht]
	\centering
	\begin{minipage}{0.6\textwidth}
		\centering
		\[
		\bordersquare{&1&2&3&4&5&6&7&8&9\cr
			&1&0&0&0&1&0&0&1&1\cr
			&0&1&0&0&1&1&0&0&1\cr
			&0&0&1&0&0&1&1&0&1\cr
			&0&0&0&1&0&0&1&1&0\cr
		}
		\]
	\end{minipage}
	\begin{minipage}{0.3\textwidth}
		\begin{tikzpicture}[scale = 1]
		\def\N{4}
		\def \angle{-360/\N}
		\draw (0,0) circle(1.4);
		
		\begin{scope}[every node/.style={shape=circle,scale=0.5,fill=white,draw}]
		\node[fill=red] (c) at (0,0) {};
		\foreach \x in {1,...,3}{
			\node[fill=red] (\x) at ($(0,0)!1!\angle*(\x-1):(1.4,0)$) {};
			\draw (c) to (\x);
		}
		\node (4) at ($(0,0)!1!\angle*(3):(1.4,0)$) {};
		\draw (c) to (4);
		\end{scope}
		\end{tikzpicture}
	\end{minipage}
	\caption{A binary representation of $P_9$ and $P_9$ as a graft of $\cW_4$.}
	\label{fig: P9}
\end{figure}

In proving Theorem~\ref{thm: 3-connected binary (1,2)}, we will require the following characterisation of binary matroids with no $M(\cW_4)$-minor due to Oxley~\cite[Theorem 2.1]{Oxley 1987}. Here $Z_r$ is the rank-$r$ binary spike with tip $t$ and $y$ is some non-tip element of $Z_r$.
\begin{lemma}\label{lem: no mW4 minor}
	Let $M$ be a binary matroid. Then $M$ is $3$-connected and has no $M(\cW_4)$ minor if and only if
	\begin{enumerate}[label=\text{\emph{(\roman*)}},leftmargin=*,align=left]
		\item $M\cong Z_r,Z_r^*,Z_r\delete y,$ or $Z_r\delete t$ for some $r\geq 3$; or
		\item $M\cong U_{0,0},U_{0,1},U_{1,1},U_{1,2},U_{1,3},$ or $U_{2,3}$.
	\end{enumerate}
\end{lemma}
The flats of the rank-$r$ binary spike are very well behaved and the straightforward proof of the following is omitted.
\begin{lemma}\label{lemma: spikes}
	For $r\geq 3$, $Z_r$ and $Z_r\delete y$ are $(2,2)$-uniform if and only if $r\leq 4$ and $Z_r\delete t$ is $(2,2)$-uniform if and only if $r\leq 5$.
\end{lemma}
Now consider the rank-$5$ binary affine geometry $AG(4,2)$. As its rank-$3$ flats are all isomorphic to $U_{3,4}$, this matroid is certainly $(2,2)$-uniform. Viewing $AG(4,2)$ as the deletion of a hyperplane $H$ from the projective geometry $PG(4,2)$, we see that every element of $H$ is in a triangle with two elements of $AG(4,2)$. It follows that any rank-$5$ binary extension of $AG(4,2)$ must have a rank-$3$ flat of nullity at least $2$ and hence fail to be $(2,2)$-uniform. Furthermore, by Lemma~\ref{lem: r,r* bound}, $AG(4,2)$ has no binary $(2,2)$-uniform coextensions. Thus, $AG(4,2)$ is a maximal binary $(2,2)$-uniform matroid. The next lemma concerning binary affine matroids will be used in the proof of Theorem~\ref{thm: 3-connected binary (1,2)}.

\begin{lemma}\label{lem: MK33 exts}
	Let $M$ be a simple rank-$5$ binary extension of $M(K_{3,3})$. Then $M$ is $(2,2)$-uniform if and only if $M$ is affine.
\end{lemma}
\begin{proof}
	If $M$ is a simple rank-$5$ binary affine matroid, then it is a restriction of $AG(4,2)$ and thus is $(2,2)$-uniform. For the other direction, let $M$ be a simple rank-$5$ binary extension of $M(K_{3,3})$ that is $(2,2)$-uniform. By uniqueness of binary representation, $M$ may be represented by a binary matrix whose first nine columns are the representation of $M(K_{3,3})$ given in Figure~\ref{fig:MK33}. Let $e$ label an extension column. It is easily seen that if the last entry of column $e$ is zero, then $e$ is in a triangle with two elements of $M(K_{3,3})$. But every pair of elements of $M(K_{3,3})$ are in a circuit of size four. Thus, if column $e$ ends in zero, then $e$ is in a rank-$3$ flat of $M$ of nullity at least $2$. This is a contradiction to the $(2,2)$-uniform property. We conclude that every extension column ends in $1$ and that, consequently, $M$ is affine.
\end{proof}

\begin{figure}[ht]
	\begin{minipage}{\textwidth}
		\centering
	\begin{minipage}{0.5\textwidth}
		\centering
	\[
	\bordersquare{&1&2&3&4&5&6&7&8&9\cr
		&1&0&0&0&0&1&0&0&1\cr
		&0&1&0&0&0&1&1&0&0\cr
		&0&0&1&0&0&0&1&1&0\cr
		&0&0&0&1&0&0&0&1&1\cr
		&1&1&1&1&1&1&1&1&1\cr
	}
	\]
	\end{minipage}
	\hspace{0.3cm}
		\begin{minipage}{0.4\textwidth}
		\begin{tikzpicture}[scale=1.1]
		\begin{scope}[every node/.style={shape=circle,scale=0.5,fill=white,draw}]
		\node (1) at (-1.5,1) {};
		\node (2) at (0,1) {};
		\node (3) at (1.5,1) {};
		\node (1') at (-1.5,-1.2) {};
		\node (2') at (0,-1.2) {};
		\node (3') at (1.5,-1.2) {};
		\draw (1) to (1') to (2) to (2') to (3) to (3') to (1) to (2');
		\draw (1') to (3);
		\draw (2) to (3');
		\end{scope}
		\end{tikzpicture}
		\end{minipage}
	\end{minipage}
	\caption{Binary and graphic representations for $M(K_{3,3})$.}
	\label{fig:MK33}
\end{figure}
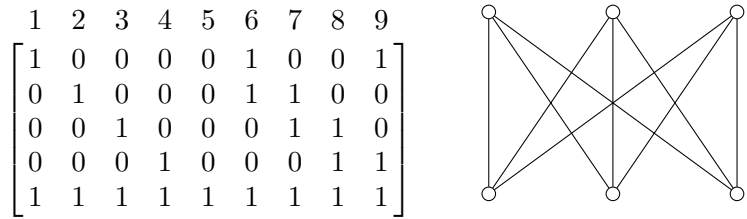

Two of the four non-isomorphic simple rank-$5$ binary single-element extensions of $M(K_{3,3})$ are affine. These are the well-known regular matroid $R_{10}$ and a matroid that we name $L_{10}$, a representation for which is given in Figure~\ref{fig:L10}. In \cite{Seymour}, $R_{10}$ is identified as the graft matroid of $K_{3,3}$ in which every vertex is coloured. We remark here that $L_{10}$ is the graft matroid of $K_{3,3}$ in which all but two vertices, both in the same partition, are coloured. 
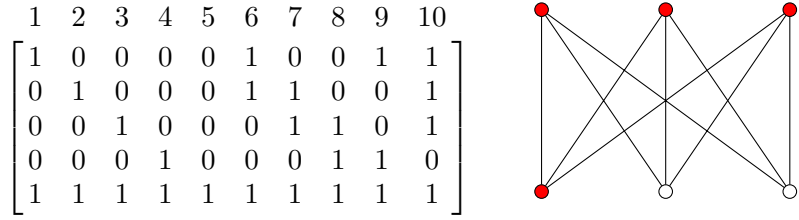
\begin{figure}[ht]
	\begin{minipage}{\textwidth}
		\centering
		\begin{minipage}{0.55\textwidth}
			\centering
			\[
			\bordersquare{&1&2&3&4&5&6&7&8&9&10\cr
				&1&0&0&0&0&1&0&0&1&1\cr
				&0&1&0&0&0&1&1&0&0&1\cr
				&0&0&1&0&0&0&1&1&0&1\cr
				&0&0&0&1&0&0&0&1&1&0\cr
				&1&1&1&1&1&1&1&1&1&1\cr
			}
			\]
		\end{minipage}
		\hspace{0.3cm}
		\begin{minipage}{0.4\textwidth}
			\begin{tikzpicture}[scale=1.1]
			\begin{scope}[every node/.style={shape=circle,scale=0.5,fill=white,draw}]
			\node[fill=red] (1) at (-1.5,1) {};
			\node[fill=red] (2) at (0,1) {};
			\node[fill=red] (3) at (1.5,1) {};
			\node[fill=red] (1') at (-1.5,-1.2) {};
			\node (2') at (0,-1.2) {};
			\node (3') at (1.5,-1.2) {};
			\draw (1) to (1') to (2) to (2') to (3) to (3') to (1) to (2');
			\draw (1') to (3);
			\draw (2) to (3');
			\end{scope}
			\end{tikzpicture}
		\end{minipage}
	\end{minipage}
	\caption{A binary representation of $L_{10}$ and $L_{10}$ as a graft of $K_{3,3}$.}
	\label{fig:L10}
\end{figure}

In our final step before proving Theorem~\ref{thm: 3-connected binary (1,2)}, we determine the binary $(2,2)$-uniform coextensions of $M(K_5\delete e)$ and $P_9$; geometric representations of which are given in Figure~\ref{fig: P9 and MK5de}.

\begin{lemma}\label{lem: L10}
	The sets of non-isomorphic binary $(2,2)$-uniform coextensions of $M(K_5\delete e)$ and $P_9$, respectively, are $\{L_{10}\}$ and $\{P_{10},L_{10}\}$.
\end{lemma}
\begin{proof}
	Let $M$ be a binary $(2,2)$-uniform matroid with a subset $X\subseteq E(M)$ such that $M/X\cong N$ for $N$ in $\{M(K_5\delete e),P_9\}$. By uniqueness of binary representation, we may assume that $M/X$ is represented by the binary matrix $A$ given in Figure~\ref{fig: matrix A}, where $\alpha \in \{0,1\}$ depends on $N$.	
	
	\begin{figure}[ht]
		\centering
		\begin{minipage}{\textwidth}
			\[
			\bordersquare{&e_1&e_2&e_3&e_4&e_5&e_6&e_7&e_8&e_9\cr
				&1&0&0&0&1&0&0&1&1\cr
				&0&1&0&0&1&1&0&0&\alpha\cr
				&0&0&1&0&0&1&1&0&1\cr
				&0&0&0&1&0&0&1&1&0\cr
			}
			\]
		\end{minipage}
		\caption{Matrix $A$. $M[A]$ is isomorphic to $M(K_5\delete e)$ when $\alpha=0$ and $P_9$ when $\alpha =1$, respectively.}
		\label{fig: matrix A}
	\end{figure}
	The set $H=\{e_1,e_2,e_3,e_5,e_6,e_9\}$ is a hyperplane of $M[A]$ regardless of $\alpha$. As $M$ is $(2,2)$-uniform, it follows that $H\cup X$ is a hyperplane of $M$ of nullity $3$. Moreover, $M|H\cup X$ is $(1,2)$-uniform and $(M|H\cup X)^*$ is $(2,1)$-uniform by duality. Thus, by Lemma~\ref{lem: simple}, $(M|H\cup X)^*$ is a rank-$3$ simple matroid with $|X|+6$ elements. It follows that $|X|=1$.  By appropriate row operations, one then sees that $M$ may be represented by the $5\times 10$ binary matrix $B$ as given in Figure~\ref{fig: matrix B}. It remains to determine the coefficients $\beta_5,\ldots,\beta_9$.

	\begin{figure}[ht]
		\centering
		\begin{minipage}{\textwidth}
			\[
			\bordersquare{&e_1&e_2&e_3&e_4&e_5&e_6&e_7&e_8&e_9&x\cr
				&1&0&0&0&1&0&0&1&1&0\cr
				&0&1&0&0&1&1&0&0&\alpha&0\cr
				&0&0&1&0&0&1&1&0&1&0\cr
				&0&0&0&1&0&0&1&1&0&0\cr
				&0&0&0&0&\beta_5&\beta_6&\beta_7&\beta_8&\beta_9&1\cr
			}
			\]
		\end{minipage}
		\caption{Matrix $B$. $M[B]/x$ is isomorphic to $M(K_5\delete e)$ when $\alpha=0$ and $P_9$ when $\alpha =1$, respectively.}
		\label{fig: matrix B}
	\end{figure}
	\begin{figure}[ht]
			\centering
			\begin{minipage}{0.4\textwidth}
				\centering
				\[
				\bordersquare{&e_5&e_6&e_9&e_1&e_2&e_3&x\cr
					&1&0&0&1&1&0&\beta_5\cr
					&0&1&0&0&1&1&\beta_6\cr
					&0&0&1&1&\alpha&1&\beta_9\cr
				}
				\]
			\end{minipage}
		\hspace{1cm}
			\begin{minipage}{0.4\textwidth}
				\centering
				\[
				\bordersquare{&e_1&e_3&e_4&x&e_7&e_8\cr
					&1&0&0&0&0&1\cr
					&0&1&0&0&1&0\cr
					&0&0&1&0&1&1\cr
					&0&0&0&1&\beta_7&\beta_8\cr
				}
				\]
			\end{minipage}
			\caption{Matrices representing $(M|H\cup x)^*$ and $M|H'\cup x$.}
			\label{fig: submats}
	\end{figure}
	A representation for $(M|H\cup x)^*$ is given in Figure~\ref{fig: submats}. As this must be simple, we deduce that $\beta_5=\beta_6=1$ and $\beta_9=1-\alpha$. To determine $\beta_7$ and $\beta_8$, we consider the hyperplane $H'=\cl_M(\{e_1,e_3,e_4\})$ of $M[A]$. If $\alpha=0$, then the hyperplane $H'\cup x$ of $M$ contains $e_9$ and by an identical argument to before, $\beta_7=\beta_8=1$. We conclude that if $N\cong M(K_5\delete e)$, then $M\cong L_{10}$. Otherwise $N\cong P_9$, $\alpha=1$ and $H'=\{e_1,e_3,e_4,e_7,e_8\}$. Then $M|H'\cup x$ is represented by the rank-$4$ matrix of Figure~\ref{fig: submats}. As this matroid must be $(1,2)$-uniform, it follows that either $\beta_7=\beta_8=1$, in which case $M\cong L_{10}$, or precisely one of $\{\beta_7, \beta_8\}$ is zero, in which case, $M\cong P_{10}$.
\end{proof}

\begin{figure}[ht]
	\centering
	\begin{minipage}{\textwidth}
		\begin{minipage}{0.5\textwidth}
			\centering
			\begin{tikzpicture}[scale=1.1]
			\begin{scope}[every node/.style=element]
			\node (1000) at (90:1.8) {};
			\node (0100) at (200:2) {};
			\node (0010) at (0:0) {};
			\node (0001) at (-20:2) {};
			
			\node (1100) at ($(1000)!0.5!(0100)$) {};
			\node (1001) at ($(1000)!0.5!(0001)$) {};

			\node(1110) at (-90:1.1) {};
			\node (0110) at (intersection of 1100--1110 and 0100--0010) {};
			\node (0011) at (intersection of 1001--1110 and 0010--0001) {};
			
			\end{scope}
			\node (e1) at ($(1000)+(0:0.4)$) {$e_1$};
			\node (e2) at ($(0100)+(-90:0.4)$) {$e_2$};
			\node (e3) at ($(0010)+(50:0.4)$) {$e_3$};
			\node (e4) at ($(0001)+(-90:0.4)$) {$e_4$};
			
			\draw (0010) to (0110) to (0100) to (1100) to (1000) to (1001) to (0001) to (0011) to (0010);
			\draw (1100) to (0110) to (1110) to (0011) to (1001);
			
			\coordinate (T) at (90:2.5);
			\coordinate (B) at (-90:1.6);
			\coordinate (TR) at ($(T)+(-30:2.7)$);
			\coordinate (BR) at  ($(B)+(-30:2.7)$);
			\coordinate (BL) at ($(B)+(210:2.7)$);
			\coordinate (TL) at ($(T)+(210:2.7)$);
			
			\draw (BL) to (TL) to (T) to (TR) to (BR) to (B) to (BL);
			\draw (T) to (B);
			
			\node[scale=1.1] at (-90:2.8) {\large{$M(K_5\delete e)$}};
			
			\end{tikzpicture}
		\end{minipage}
		\begin{minipage}{0.5\textwidth}
			\centering
			\begin{tikzpicture}[scale=1.1]
			\begin{scope}[every node/.style=element]
			\node (1000) at (90:1.8) {};
			\node (0010) at (0,-0.5) {};
			\node (0100) at ($(0010)+(200:2)$) {};
			\node (0001) at ($(0010)+(-20:2)$) {};

			\node (1100) at ($(1000)!0.5!(0100)$) {};
			\node (1001) at ($(1000)!0.5!(0001)$) {};
			\node (0110) at ($(0100)!0.5!(0010)$) {};
			\node (0011) at ($(0010)!0.5!(0001)$) {};
			
			\coordinate(1110) at (-90:0.9) {};
			
			\node (1010) at (intersection of 1000--0110 and 0010--1100) {};
			
			\coordinate (C) at (intersection of 0100--1010 and 1000--0010) {};
			
			\coordinate (C') at ($(C)+(-90:0.1)$) {};
			
			\end{scope}
			\node (e1) at ($(1000)+(0:0.4)$) {$e_1$};
			\node (e2) at ($(0100)+(-90:0.4)$) {$e_2$};
			\node (e3) at ($(0010)+(-60:0.4)$) {$e_3$};
			\node (e4) at ($(0001)+(-90:0.4)$) {$e_4$};

			\draw (0010) to (0110) to (0100) to (1100) to (1000) to (1001) to (0001) to (0011) to (0010);
			\draw (1000) to (1010) to (0110);
			\draw (1100) to (1010) to (0010);
			\draw[dashed] (0100) to (1010) to (C);
			
		    \draw[dashed] plot [smooth,tension=1] coordinates{(1100) (C') (0110)};

			\draw[dashed] plot [smooth,tension=1] coordinates{(0011) (C') (1001)};

			\coordinate (T) at (90:2.5);
			\coordinate (B) at (-90:1.6);
			\coordinate (TR) at ($(T)+(-30:2.7)$);
			\coordinate (BR) at  ($(B)+(-30:2.7)$);
			\coordinate (BL) at ($(B)+(210:2.7)$);
			\coordinate (TL) at ($(T)+(210:2.7)$);
			
			\draw (BL) to (TL) to (T) to (TR) to (BR) to (B) to (BL);
			\draw (T) to (B);
			
			\node[scale=1.1] at (-90:2.8) {\Large{$P_9$}};
			
			\end{tikzpicture}
		\end{minipage}
	\end{minipage}
	\caption{Geometric representations of $M(K_5\delete e)$ and $P_9$.}
	\label{fig: P9 and MK5de}
\end{figure}
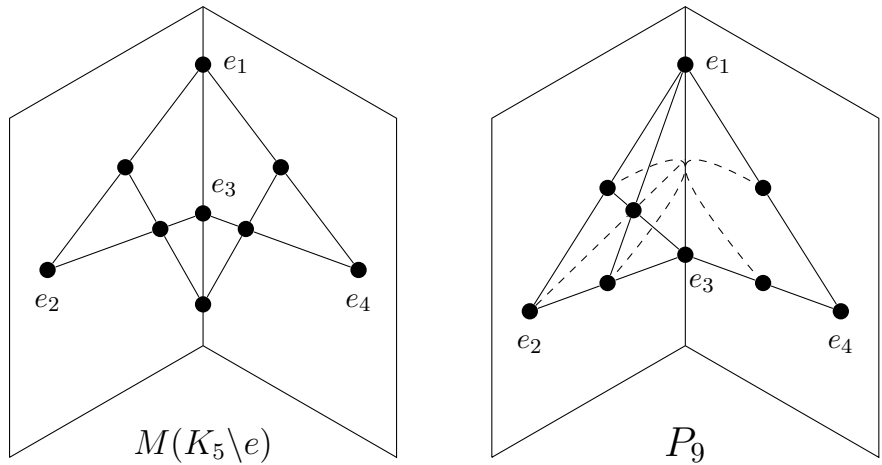

We now conclude the paper by proving Theorem~\ref{thm: 3-connected binary (1,2)}.
\begin{proof}[Proof of Theorem~\ref{thm: 3-connected binary (1,2)}]
	We first observe that a matroid is a binary $3$-connected $(2,2)$-uniform matroid if and only if its dual is also. In particular, both $AG(4,2)$ and $AG(4,2)^*$ are minor-maximal such matroids. To complete our list, let $M$ be a minor-maximal binary $3$-connected $(2,2)$-uniform matroid. If $r(M)\leq 4$, or $r^*(M)\leq 4$, then $M$ is a minor of either $AG(4,2)$ or $AG(4,2)^*$, a contradiction to maximality. Thus, $r(M),r^*(M)\geq 5$. Switching to the dual if necessary, we may then assume by Lemma~\ref{lem: r,r* bound} that $r(M)=5$.
	
	If $M$ has no $M(\cW_4)$ minor, then by Lemma~\ref{lem: no mW4 minor}, $M$ is isomorphic to one of $Z_r, Z_r^*, Z_r\delete t, Z_r \delete y$ for some $r\geq 3$ and, by Lemma~\ref{lemma: spikes}, $M\cong Z_5\delete t$. Otherwise, we may assume that $M$ does possess an $M(\cW_4)$-minor. Then, as $r(M)=~5$, $M$ is an extension of a single-element coextension $N$ of $M(\cW_4)$.  As $M(\cW_4)$ is self-dual, the matroid $N^*$ is a binary $(2,2)$-uniform single-element extension of $M(\cW_4)$. These are just the simple binary extensions of $M(\cW_4)$, namely $M(K_5\delete e)$, $P_9$ and $M^*(K_{3,3})$. Thus, $N\in \{M^*(K_5\delete e), P_9^*, M(K_{3,3})\}$. If $N\cong~ M(K_{3,3})$, then by Lemma~\ref{lem: MK33 exts}, $M$ must be affine and thus, by maximality, $M\cong AG(4,2)$. Otherwise, $N\in \{M^*(K_5\delete e), P_9^*\}$, in which case, by the dual of Lemma~\ref{lem: L10}, $M$ is isomorphic to either $P_{10}$ or $L_{10}^*$. But, as $L_{10}$ is affine, $L_{10}^*$ is a minor of $AG(4,2)^*$. We conclude by maximality that, in this case, $M\cong P_{10}$. The theorem then follows by duality.
\end{proof}

\section*{Acknowledgements}
The author thanks Professors Charles Semple and James Oxley for reading the early drafts of this paper and for their invaluable suggestions towards improving the exposition.

\end{document}